\documentclass[english]{amsart}%
\usepackage{amsfonts}
\usepackage{amsmath}
\usepackage{amssymb}
\usepackage{amsthm}
\usepackage{amscd}
\usepackage{babel}
\usepackage[latin1]{inputenc}
\usepackage{graphicx}%
\providecommand{\U}[1]{\protect\rule{.1in}{.1in}}
\newtheorem{teor}{Theorem}

\newtheorem{prop}{Proposition}

\theoremstyle{definition}

\newtheorem*{rem}{Remark}

\renewcommand{\subjclassname}{AMS \textup{2010} Mathematics Subject
Classification\ }
\email{bayon@uniovi.es}
\email{fortunypedro@uniovi.es}
\email{grau@uniovi.es}
\email{oller@unizar.es}
\email{mruiz@uniovi.es}

\begin{document}
\author{L. Bayón}
\address{Departamento de Matemáticas, Universidad de Oviedo\\
Avda. Calvo Sotelo s/n, 33007 Oviedo, Spain}
\author{P. Fortuny Ayuso}
\address{Departamento de Matemáticas, Universidad de Oviedo\\
Avda. Calvo Sotelo s/n, 33007 Oviedo, Spain}
\title{The Best-or-Worst and the Postdoc problems}
\author{J.M. Grau}
\address{Departamento de Matemáticas, Universidad de Oviedo\\
Avda. Calvo Sotelo s/n, 33007 Oviedo, Spain}
\author{A. M. Oller-Marcén}
\address{Centro Universitario de la Defensa de Zaragoza\\
Ctra. Huesca s/n, 50090 Zaragoza, Spain}
\author{M.M. Ruiz}
\address{Departamento de Matemáticas, Universidad de Oviedo\\
Avda. Calvo Sotelo s/n, 33007 Oviedo, Spain}

\begin{abstract}
We consider two variants of the secretary problem, the\emph{ Best-or-Worst} and the \emph{Postdoc} problems, which are closely related. First, we prove that both variants, in their standard form with binary
payoff 1 or 0, share the same optimal stopping rule. We also consider additional
cost/perquisites depending on the number of interviewed candidates. In these situations the optimal strategies are very different. Finally, we also focus on the Best-or-Worst variant with different payments depending on whether the selected candidate is the best or the worst.

\end{abstract}
\maketitle

\subjclassname{60G40, 62L15}

\section{Introduction}

The \emph{secretary problem} is one of many names for a famous problem of optimal
stopping theory. This problem can be stated as follows: an
employer is willing to hire the best secretary out of $n$ rankable candidates. These candidates are interviewed one by one in random order. A
decision about each particular candidate is to be made immediately after the
interview. Once rejected, a candidate cannot be called back. During the
interview, the employer can rank the candidate among all the preceding ones, but he is unaware of the quality of yet unseen candidates.
The goal is then to determine the optimal strategy that maximizes the
probability of selecting the best candidate.

This problem has a very elegant solution. Dynkin \cite{48} and Lindley
\cite{101} independently proved that the best strategy consists in a so-called threshold strategy. Namely, in rejecting
roughly the first $n/e$ (cutoff value) interviewed candidates and then choosing the first one that is better
than all the preceding ones. Following this strategy, the
probability of selecting the best candidate is at least $1/e$, this being its approximate value for large
values of $n$. This well-known solution was later refined by Gilbert and Mosteller \cite{gil} showing that $\left\lfloor (n-\frac{1}{2})e^{-1}+\frac{1}%
{2}\right\rfloor$ is a better approximation than $\lfloor n/e\rfloor$, although the difference
is never greater than 1.

This secretary problem has been addressed by many authors in different fields such as applied probability, statistics or decision theory. In \cite{FER}, \cite{FER2} or \cite{2009} extensive
bibliographies on the topic can be found. On the other hand, different generalizations of this classical problem have been recently considered in the framework of partially ordered objects \cite{poset2,garrod,poset1} or matroids \cite{1,soto}. It is also worth mentioning the work of Bearden \cite{KK}, where the author considers a situation where the employer receives a payoff for selecting a candidate equal to the ``score'' of the candidate (in the classical problem the payoff is 1 if the candidate is really the best and 0 otherwise). In this situation, the optimal cutoff value is roughly the square root of the number of candidates.

In this paper we focus on two closely related variants of the secretary problem. The so-called  \emph{Best-or-Worst} and \emph{Postdoc} variants. In the Best-or-Worst variant, the classic secretary problem is modified so that the goal is to select either the best or the worst candidate, indifferent between the two cases. This variant can only be found on \cite{fergu} as a multicriteria problem in the perfect negative dependence case. Here we present it in greater detail. In the Postdoc variant, instead of selecting the best candidate, the goal is to select the second best candidate. This problem was proposed to Robert J. Vanderbei by Eugene Dynkin in 1980 with the following motivating story that explains the name of the problem: we are trying to hire a
postdoc, since the best candidate will receive (and accept) an offer from Harvard, we are interested in hiring the second best candidate. Vanderbei himself solved the problem in 1983 using dynamic programming \cite{posdoc}. However, he never published his work because he learned that Rose had already published his own solution using different techniques \cite{rose}. Moreover, Szajowski had already solved the problem of picking the $k$-th better candidate for $2\leq k\leq 5$ \cite{aesima}.

In the present paper, for these two variants, we study the standard problem (binary payoff function 1 or 0), showing that both have the same optimal
cutoff rule strategy and also the problems considering payoff functions that depend on the number of performed interviews, showing that in this case they have very different optimal strategies.

The paper is organized as follows: in Section 2, we present some technical results, in Section 3, we revisit the classic secretary problem and also solve two new situations with payoff functions that depend on the number of performed interviews. In Section 4 we focus on the Best-or-Worst variant, solving the problem for three different payoff functions and also presenting a variant in which the choice of the best or the worst candidate is no longer indifferent. In Section 5 we solve the three versions of the Postdoc variant and, finally, we compare the obtained results in Section 6.

\section{Two technical results}

The following result can be widely applied in different optimal stopping problems and it will be extensively used throughout the paper. For a sequence of continuous real functions $\{F_{n}\}_{n\in\mathbb{N}}$ defined on a closed interval, it determines the asymptotic behavior of the sequence $\{\mathcal{M}(n)\}_{n\in\mathbb{N}}$, where $\mathcal{M}(n)$ is the value for which the function $F_n$ reaches its maximum.

\begin{prop}\label{conv}
Let $\{F_{n}\}$ be a sequence of real functions with $F_n\in\mathcal{C}[0,n]$ and let $\mathcal{M}(n)$ be the value for which the function $F_n$ reaches its maximum. Assume that the sequence of functions $\{g_{n}\}_{n\in\mathbb{N}}$ given by $g_{n}(x):=F_{n}(nx)$ converges uniformly on
$[0,1]$ to a function $g$ and that $\theta$ is the only global maximum of $g$ in $[0,1]$. Then,
\begin{itemize}
\item[i)] $\displaystyle\lim_{n} \mathcal{M}(n)/n =\theta$.
\item[ii)] $\displaystyle\lim_{n} F_{n}(\mathcal{M}(n))= g(\theta)$.
\item[iii)] If $\mathfrak{M}(n)\sim\mathcal{M}(n)$ then $\displaystyle\lim_{n}F_{n}(\mathfrak{M}(n))=g(\theta)$.
\end{itemize}
\end{prop}
\begin{proof}
\begin{itemize}
\item[i)]
Let us consider the sequence $\{\mathcal{M}(n)/n\}\subset [0,1]$ and assume that $\{\mathcal{M}(s_{n})/s_{n}\}$ is a subsequence that converges to $\alpha$.
Then,
$$g_{s_n}(\theta)=F_{s_n}(s_n\theta)\leq F_{s_{n}}(\mathcal{M}(s_{n}))=F_{s_{n}}\left(\frac{\mathcal{M}(s_{n})}{s_n}s_n\right)=g_{s_n}\left(\frac{\mathcal{M}(s_n)}{s_n}\right).$$
Consequently, since $g_n\to g$ uniformly on $[0,1]$, if we take limits we get
$$g(\theta)=\lim_n g_{s_n}(\theta)\leq\lim_n g_{s_n}\left(\frac{\mathcal{M}(s_n)}{s_n}\right)=g(\alpha)$$
and since $\theta$ is the only global maximum of $g$, it follows that $\theta=\alpha$.

Thus, we have proved that every convergent subsequence of $\{\mathcal{M}(n)/n\}$ converges to the same limit $\theta$. Since $\{\mathcal{M}(n)/n\}$ is defined on a compact set this implies that $\{\mathcal{M}(n)/n\}$ itself must also converge to $\theta$.
\item[ii)] It is enough to observe that
$$\lim_{n} F_{n}(\mathcal{M}(n))=\lim_{n}F_{n}\left(\frac{\mathcal{M}(n)}{n} n\right)=\lim_n g_n\left(\frac{\mathcal{M}(n)}{n}\right)=g(\theta),$$
where the last equality holds because $g_n\to g$ uniformly on $[0,1]$.
\item[iii)] If $\mathfrak{M}(n)\sim\mathcal{M}(n)$, then it also holds that $\displaystyle \lim_n \frac{\mathfrak{M}(n)}{n}=\theta$ and we can reason as in the previous point.
\end{itemize}
\end{proof}

\begin{rem}
The condition of uniform convergence is required to ensure, for instance, that $\displaystyle \lim_n g_{s_n}\left(\frac{\mathcal{M}(s_n)}{s_n}\right)=g(\alpha)$. In fact, it is easy to give counterexamples to Proposition \ref{conv} if convergence is not uniform.
\end{rem}

Observe that Proposition \ref{conv} implies that that $\displaystyle \lim_n F_{n}(n \theta)=g(\theta)$. Moreover, it also implies that $\displaystyle \lim_n F_{n}(n \theta+o(n))=g(\theta)$. This means that $n\theta$ is a good estimate for $\mathcal{M}(n)$ and that, for large values of $n$, the maximum value of $F_n$ approaches $g(\theta)$.

Proposition \ref{conv} admits the following two-variable version that can be proved in the same way.

\begin{prop}\label{conv2}
Let $\{G_{n}\}$ be a sequence of two variable real functions with $G_{n}\in\mathcal{C}\big(\{(x,y)\in\lbrack0,n]^{2}:x\leq y\}\big)$ and let
$(\mathcal{M}_{1}(n),\mathcal{M}_{2}(n))$ be a point for which $G_{n}$ reaches its maximum. Assume that the sequence $\{h_n\}_{n\in\mathbb{N}}$ given by $h_{n}(x,y):=G_{n}(nx,ny)$ converges uniformly on $T:=\{(x,y)\in\mathbb{R}^2:0\leq x\leq y\leq 1\}$ to a function $h$ and that $(\theta_1,\theta_2)$ is the only global maximum of $h$ in $T$. Then,
\begin{itemize}
\item[i)] $\displaystyle\lim_{n} \mathcal{M}_{i}(n)/n =\theta_{i}$ for $i=1,2$.
\item[ii)] $\displaystyle\lim_{n} G_{n}(\mathcal{M}_{1}(n),\mathcal{M}_{2}(n))=
h(\theta_{1},\theta_{2}).$
\item[iii)] If $\mathfrak{M}_{i}(n)\sim\mathcal{M}_{i}(n)$ for $i=1,2$, then $\displaystyle\lim_{n} G_{n}( \mathfrak{M}_{1}(n),\mathfrak{M}_{2}(n))=h(\theta_{1},\theta_{2})$.
\end{itemize}
\end{prop}

\section{A new look at the classic secretary problem}

In the classical secretary problem, let $n$ be the number of candidates and let us consider a cutoff value $r\in(1,n)$. If $k\in (r,n]$ is an integer, the probability of successfully selecting the best candidate in the $k$-th interview is $\displaystyle P_{n,r}(k)=\frac{r}{n}\frac{1}{k-1}$. Thus, the probability function of succeeding in the classical secretary problem with $n$ candidates using $r$ as cutoff value, is given by
$$F_{n}(r):=\sum_{k=r+1}^{n}P_{n,r}(k)=\frac{r}{n}\sum_{k=r+1}^{n}\frac{1}{k-1}.$$

The goal is now to determine the value of $r$ that maximizes this probability (i.e., to determine the optimal cutoff value) and to compute this maximum probability. This can be done using Proposition \ref{conv} in the following way. First, we extend $F_{n}$ to a real variable function by
$$F_{n}(r)=\frac{r}{n}(\psi(n)-\psi(r)),$$
where $\psi$ is the so-called digamma function. Then, it can be seen with little effort that the sequence of
functions $\{g_n\}$ defined by $g_{n}(x):=F_{n}(nx)$ converges uniformly on $[0,1]$ to the function
$g(x):=-x\log(x)$ and the remaining is just some elementary calculus.

\begin{rem}
In \cite{FER} the following rather lax reasoning showing that $\mathcal{M}(n)/n$ tends to $1/e$ is given. If we let $n$ tend to infinity and write $x$ as the limit of $r/n$, then
using $t$ for $j/n$ and $dt$ for $1/n$, the sum becomes a Riemann approximation to an integral
$$F_{n}(r) \rightarrow x \int_{x}^{1} \frac{dt}{t}= - x \log(x).$$
Proposition 1 provides a more rigorous approach.
\end{rem}

We introduce a more general situation. Let $p:\mathbb{R}\to[0,+\infty)$ be a function (payoff function) and assume that a payoff of $p(k)$ is received if the $k$-th candidate is selected. In this setting, the expected payoff is
$$E_n(r):=\sum_{k=r+1}^n p(k)P_{n,r}(k)=\frac{r}{n}\sum_{k=r+1}^{n}\frac{p(k)}{k-1}.$$
Note that in the classical situation
\begin{equation}\label{pobin}
p_B(k)=\begin{cases} 1, & \textrm{if the $k$-th candidate is the seeked candidate};\\ 0, & \textrm{otherwise}.\end{cases}
\end{equation}
and the expected payoff coincides with the probability of successfully selecting the best candidate.

Now, let us modify the classical situation considering that performing each interview has a constant cost of $1/n$. Clearly, in this situation the payoff function is given by
\begin{equation}\label{pocost}
p_C(k)=\begin{cases} 1-k/n, & \textrm{if the $k$-th candidate is the seeked candidate};\\ 0, & \textrm{otherwise}.\end{cases}
\end{equation}
and the expected payoff is
$$E^{C}_{n}(r):=\frac{r}{n}\sum_{k=r+1}^{n}\frac{1-\frac{k}{n}}{k-1}.$$
The following result provides the optimal cutoff value and the maximum expected payoff in this setting. In what follows, we denote by $W$ the main branch of the so-called Lambert-$W$ function, defined by $z=W(ze^z)$.

\begin{prop}
Given an integer $n>1$, let us consider the function
$$E^{C}_{n}(r):=\frac{r}{n}\sum_{k=r+1}^{n}\frac{1-\frac{k}{n}}{k-1}$$
defined for every integer $1\leq r\leq n-1$ and let $\mathcal{M}(n)$ be the value for which the function $E^{C}_{n}$ reaches its
maximum. Then,
\begin{itemize}
\item[i)] $\displaystyle\lim_{n} {\mathcal{M}(n)}/{n}=\rho:= -\frac{1}{2}W(-2 e^{-2}) =0.20318\dots$.
\item[ii)] $\displaystyle \lim_{n}E^{C}_{n}(\mathcal{M}(n))=\displaystyle \lim_{n}E^{C}_{n}(\lfloor \rho n\rfloor)=\rho(1-\rho)= 0.16190\dots$.
\end{itemize}
\end{prop}
\begin{proof}
First, we extend $E^{C}_{n}$ to a real variable function by
$$E^C_{n}(r)=\frac{r\,\left(  -n+r+\left(  -1+n\right)  \,\psi(n)-\left(-1+n\right)  \,\psi(r)\right)  }{n^{2}}.$$
Now, it can be seen that $g_{n}(x):=E^{C}_{n}(nx)$ converges uniformly in $[0,1]$ to $g(x):=x\left(-1+x-\log(x)\right)$.
To conclude the proof it is enough to apply Proposition \ref{conv} together with some straightforward computations.
\end{proof}

This result means that the optimal strategy in this setting consists in rejecting roughly the first $\rho n$ interviewed candidates and then accepting the first candidate which is better than all the preceding ones. Following this strategy, the maximum expected payoff is asymptotically equal to $\rho^{2}-\rho$.

\begin{rem}
The constant $\rho=-\frac{1}{2}W(-2e^{-2})=0.20318786\dots$ (A106533
in OEIS) appears in \cite{FER2} (erroneously approximated as 0.20388) in the context of the Best-Choice Duration
Problem considering a payoff of $(n-k+1)/n$. Furthermore, as a noteworthy curiosity, it should be pointed out that this constant has
appeared in a completely different context from the one addressed here (the Daley-Kendall model) and it is known as the
\emph{rumour's constant} \cite{RUMOR,ru}.
\end{rem}

Now, let us consider that performing each interview has an perquisite of $1/n$. Clearly, in this situation the payoff function is given by
\begin{equation}\label{popay}
p_P(k)=\begin{cases} 1+k/n, & \textrm{if the $k$-th candidate is the seeked candidate};\\ 0, & \textrm{otherwise}.\end{cases}
\end{equation}
and the expected payoff is
$$E^{P}_{n}(r):=\frac{r}{n}\sum_{k=r+1}^{n}\frac{1+\frac{k}{n}}{k-1}.$$
The following result provides the optimal cutoff value and the maximum expected payoff in this setting.

\begin{prop}
Given an integer $n>1$, let us consider the function
$$E^{P}_{n}(r):=\frac{r}{n}\sum_{k=r+1}^{n}\frac{1+\frac{k}{n}}{k-1}$$
defined for every integer $1\leq r\leq n-1$ and let
$\mathcal{M}(n)$ be the value for which the function $E^{P}_{n} $ reaches its
maximum. Then,
\begin{itemize}
\item[i)] $\displaystyle\lim_{n} {\mathcal{M}(n)}/{n}=\mu:= \frac{1}{2}W( 2 ) =0.42630\dots$.
\item[ii)] $\displaystyle \lim_{n}E^{P}_{n}(\mathcal{M}(n))=\displaystyle \lim_{n}E^{P}_{n}(\lfloor\mu n \rfloor)=\mu(1+\mu)= 0.608037\dots$.
\end{itemize}
\end{prop}
\begin{proof}
First, we extend $E^{P}_{n}$ to a real variable function by
$$E^{P}_{n}(r)=\frac{r\,\left(  n - r + \left(  1 + n \right)  \,\psi(n) - \left(  1+ n \right)  \,\psi(r) \right)  }{n^{2}}.$$
Now it can be seen that $g_{n}(x):=E^{P}_{n}(nx)$ converges uniformly in $[0,1]$ to $g(x):=-x\left(-1+x+\log(x)\right)$.
To conclude the proof it is enough to apply Proposition \ref{conv} together with some straightforward computations.
\end{proof}

This result means that the optimal strategy in this setting consists in rejecting roughly the first $\mu n$ interviewed candidates and then accepting the first candidate which is better than all the preceding ones. Following this strategy, the maximum expected payoff is asymptotically equal to $\mu^{2}+\mu$.

\section{The Best-or-Worst variant}

In this section we focus on the Best-or-Worst variant, as described in the introduction, in which the goal is to select either the best or the worst candidate, indifferent between the two cases. First of all we prove that, just like in the classic problem, the optimal strategy is a threshold strategy.

\begin{teor}\label{BWS}
For the Best-or-Worst variant, if $n$ is the number of objects, there exists $r(n)$ such
that the following strategy is optimal:
\begin{enumerate}
\item Reject the $r(n)$ first interviewed candidates.
\item After that, accept the first candidate which is either better or worse than all the preceding ones.
\end{enumerate}
\end{teor}
\begin{proof}
For the sake of brevity, a candidate which is either better or worse than all the preceding ones will be called a \emph{nice candidate}.

Since the game under consideration is finite, there must exist an optimal strategy (in the sense that it
maximizes the probability of success). Hence, we can define $P_{rej}(k)$ as
the probability of success following an optimal strategy when rejecting a
candidate in the $k$-th interview (regardless of its being a nice
candidate or not). We can also define $P_{acc}(k)$ as the probability of success
accepting a nice candidate in the $k$-th interview. Any optimal
strategy will reject any non-nice candidate since the probability of being a
successful choice will be $0$.

Probability $P_{acc}(k)$ is $k/n$, which
increases with $k$. On the other hand, the function $P_{rej}(k)$ is non-increasing because
$$P_{rej}(k)=p\cdot(\max\{P_{acc}(k+1),P_{rej}(k+1)\}+(1-p)P_{rej}(k+1)\geq
P_{rej}(k+1).$$
Thus, since $P_{acc}$ is increasing and $P_{rej}$ is
non-increasing and given that $P_{{acc}}(n)=1$ and $P_{rej}(n)=0$, there exists a
natural number $r(n)$ for which:
$$P_{acc}(k)<P_{rej}(k)\ \textrm{if}\ k\leq r(n),$$
$$P_{acc}(k)\geq P_{rej}(k)\ \textrm{if}\ k>r(n).$$

As a consequence of this fact, the following strategy must be optimal: for each $k$-th interview with $k\in\{1,\dots, n\}$ do the following:
\begin{itemize}
\item Reject the $k$-th candidate if $k\leq r(n)$ or if it is not a nice candidate.
\item Accept the $k$-th candidate if $k>r(n)$ and it is a nice candidate.
\end{itemize}
Note that the optimality of this strategy follows from the fact that, in each interview, we are choosing the action with greatest probability of success.
\end{proof}

Once that we have determined the optimal strategy, we focus on determining the probability of success in the $k$-th interview. To do so, let $n$ be the number of candidates and let us consider a cutoff value $r\in(1,n)$. If $k\in (r,n]$ is an integer, the probability of successfully selecting the best or the worst candidate in the $k$-th interview is  $\displaystyle P^{BW}_{n,r}(k)=\frac{2}{n}\frac{\binom{r}{2}}{\binom{k-1}{2}}$. Thus, the probability function of succeeding in the Best-or-Worst variant with $n$ candidates using $r$ as cutoff value, is given by
$$F^{BW}_{n}(r):=\sum_{k=r+1}^{n}P^{BW}_{n,r}(k)=\frac{2r(r-1)}{n}\sum_{k=r+1}^{n}\frac{1}{(k-1)(k-2)}=\frac{2r(n-r)}{n(n-1)},$$
where the last equality follows using telescopic sums.

\begin{rem}
Note that for $n>r\in\{0,1\}$, it is straightforward to see that the
probability of success is
$$F^{BW}_{n}(0)=F^{BW}_{n}(1)=\frac{2}{n}.$$
\end{rem}

The goal is now to determine the value of $r$ that maximizes the probability $F^{BW}_{n}$ (i.e., to determine the optimal cutoff value) and to compute this maximum probability. We do so in the following result.

\begin{teor}\label{BWP}
Given a positive integer $n>2$, let us consider the function
$$F^{BW}_{n}(r)=\frac{2r(n-r)}{n(n-1)}$$
defined for every integer $2\leq r\leq n-1$ and let $\mathcal{M}(n)$ be the
value for which the function $F^{BW}_{n}$ reaches its maximum. Then,
\begin{itemize}
\item[i)] $\mathcal{M}(n)=\lfloor n/2\rfloor$.
\item[ii)] The maximum value of $F^{BW}_{n}$ is:
$$F^{BW}_{n}(\mathcal{M}(n))=\frac{\lfloor\frac{1+n}{2}\rfloor
}{2\lfloor\frac{1+n}{2}\rfloor-1}=%
\begin{cases}
\frac{n}{2(n-1)}, & \text{if $n$ is even};\\
\frac{n+1}{2n}, & \text{if $n$ is odd}.
\end{cases}
$$
\end{itemize}
\end{teor}
\begin{proof}
\

\begin{itemize}
\item[i)] Since $F^{BW}_{n}(r)=-\frac{2}{n(n-1)}r^{2}+\frac{2}{(n-1)}r$ is the
equation of a parabola in the variable $r$, it is clear that
$$\mathcal{M}(n)=\min\left\{  r\in[2,n-1]:F^{BW}_{n}(r)\geq F^{BW}_{n}(r+1)\right\}.$$
Now,
$$F^{BW}_{n}(r+1)-F^{BW}_{n}(r)=\frac{2}{n(n-1)}(n-2r-1)$$
so it follows that
$$F^{BW}_{n}(r+1)-F^{BW}_{n}(r)\leq0\Leftrightarrow(n-2r-1)\leq0\Leftrightarrow r\geq
\frac{n-1}{2}.$$
Consequently,
$$\mathcal{M}(n)=\min\left\{  r\in[2,n-1]:r\geq\frac{n-1}{2}\right\}  =\lfloor
n/2\rfloor$$
as claimed.
\item[ii)] It is enough to apply the previous result.

If $n$ is even, then $n=2N$ and
$$F^{BW}_{n}(\mathcal{M}(n))=F^{BW}_{n}(N)=\frac{2N(n-N)}{n(n-1)}=\frac{2N^{2}}
{2N(2N-1)}=\frac{N}{2N-1}.$$
Moreover, in this case
$$\left\lfloor \frac{1+n}{2}\right\rfloor =\left\lfloor \frac{1+2N}
{2}\right\rfloor =N$$
so it follows that
$$F^{BW}_{n}(\mathcal{M}(n))=\frac{N}{2N-1}=\frac{\left\lfloor \frac{1+n}
{2}\right\rfloor }{2\left\lfloor \frac{1+n}{2}\right\rfloor -1}$$
as claimed.

Otherwise, if $n$ is odd, then $n=2N+1$ and
$$F^{BW}_{n}(\mathcal{M}(n))=F_{n}^{BW}(N)=\frac{2N(n-N)}{n(n-1)}=\frac{2N(2N+1-N)}
{(2N+1)2N}=\frac{N+1}{2N+1}.$$
In this case
$$\left\lfloor \frac{1+n}{2}\right\rfloor =\left\lfloor \frac{1+2N+1}
{2}\right\rfloor =N+1$$
so we also have that
$$F^{BW}_{n}(\mathcal{M}(n))=\frac{N+1}{2N+1}=\frac{\left\lfloor \frac{1+n}
{2}\right\rfloor }{2\left\lfloor \frac{1+n}{2}\right\rfloor -1}$$
and the proof is complete.
\end{itemize}
\end{proof}

This result means that, for $n>2$, optimal strategy in this setting consists in rejecting roughly the first $\lfloor\frac{n}{2}\rfloor$ interviewed candidates and then accepting the first candidate which is either better or worse than all the preceding ones. Following this strategy, the maximum probability of success is $\displaystyle\frac{\lfloor\frac{1+n}{2}\rfloor}{2\lfloor\frac{1+n}{2}\rfloor-1}$. In the cases $n\in\{1,2\}$, it is evident that an optimal cutoff value is $r=0$, i.e. to accept the first candidate that we consider The probability of success is 1 in both cases according to the fact that $F^{BW}_1(0)=F^{BW}_2(0)=1$.

\begin{rem}
Unlike in the classic secretary problem, the probability of success in the Best-or-Worst variant is not strictly increasing in $n$. In fact, we have that $F^{BW}_{2n}(\mathcal{M}(2n))=F^{BW}_{2n-1}(\mathcal{M}(2n-1))$ for every $n$.
\end{rem}

We are now going to consider the Best-or-Worst variant with the payoff function $p_C$ given in (\ref{pocost}); i.e., we assume that performing each interview has a constant cost of $1/n$. Under this assumption it can be proved that the optimal strategy is the same threshold strategy given in Theorem \ref{BWS}. Moreover, in this setting, the expected payoff with $n$ candidates and cutoff value
$r$ is given by
$$E_n^{BW,C}(r):=\sum_{k=r+1}^n\left(1-\frac{k}{n}\right)P^{BW}_{n,r}(k)=\frac{2r(r-1)}{n^2}\sum_{k=r+1}^{n}\frac{n-k}{(k-1)(k-2)}.$$
As usual, the goal is to determine the optimal cutoff value that maximizes the expected payoff $E^{BW,C}_{n}$ and to compute this maximum expected payoff. We do so in the following result.

\begin{teor}
Given an integer $n>1$, let us consider the function $E^{BW,C}_n(r)$ defined above for every integer
$1<r<n$ and let $\mathcal{M}(n)$ be the value for which the function $E^{BW,C}_n$ reaches its maximum. Also, let
$$\theta:=-\frac{1}{2W_{_{-1}}(-\frac{1}{2\sqrt{e}})}=e^{\frac{1}{2} + W_{-1}(\frac{-1}{2\,\sqrt{e}})}$$
be the solution to the equation $2x\log(x)=x-1$. Then, the following hold:
\begin{itemize}
\item[i)] $\displaystyle\lim_{n} {\mathcal{M}(n)}/{n}=\theta= 0.284668\dots$.
\item[ii)] $\displaystyle \lim_{n}E^{BW,C}_n( \mathcal{M}(n))=\displaystyle \lim_{n}E^{BW,C}_n(\lfloor n  \theta
\rfloor)=\theta(1-\theta)=0.2036321\dots$
\end{itemize}
\end{teor}
\begin{proof}
First, observe that
\begin{align*}
E^{BW,C}_n(r)&=\frac{2r(r-1)}{n^{2}}\sum_{k=r+1}^{n}\frac{(n-k)}{(k-1)(k-2)}%
=\frac{2r(r-1)}{n^{2}}\left[  \frac{n-2}{r-1}-\frac{n-2}{n-1}-\sum_{i=r}%
^{n-1}\frac{1}{i}\right]\\
&  =2\frac{r}{n}\left(  1-\frac{2}{n}\right)  -2\frac{r}{n}\left(  \frac
{r}{n-1}-\frac{1}{n-1}\right)  -2\frac{r}{n}\left(  \frac{r}{n}-\frac{1}%
{n}\right)  \sum_{i=r}^{n-1}\frac{1}{i}.
\end{align*}
Now, we can extend $E^{BW,C}_n$ to a real variable function by
$$E^{BW,C}_n(r)=2\frac{r}{n}\left(  1-\frac{2}{n}\right)  -2\frac{r}{n}\left(  \frac
{r}{n-1}-\frac{1}{n-1}\right)  -2\frac{r}{n}\left(  \frac{r}{n}-\frac{1}%
{n}\right) (\psi(n)-\psi(r)).$$
Furthermore, it can be seen that the sequence of functions $g_n(x):=E^{BW,C}_n(nx)$ converges uniformly
in $[0,1]$ to the function $g(x)=2x\left(1-x+x\log x\right)$.

To conclude the proof it is enough to apply Proposition \ref{conv} together with some straightforward computations.
\end{proof}

\begin{rem}
The constant $\theta=-\frac{1}{2W_{-1}(-\frac{1}{2\sqrt{e}})}=0.284668\dots$
also appears related to rumour theory \cite{RUMOR,ru} and to Gabriel's Horn
(see A101314 in OEIS).
\end{rem}
%

Now, let us consider the Best-or-Worst variant with the payoff function $p_P$ given in (\ref{popay}); i.e., we assume that performing each interview has an additional payoff of $1/n$. Under this assumption, since the payoff increases with the number of interviews, it can be proved that the optimal strategy is again the same threshold strategy given in Theorem \ref{BWS}. Moreover, in this setting, the expected payoff with $n$ candidates and cutoff value
$r$ is given by
$$E_n^{BW,P}(r):=\sum_{k=r+1}^n\left(1+\frac{k}{n}\right)P^{BW}_{n,r}(k)=\frac{2r(r-1)}{n^2}\sum_{k=r+1}^{n}\frac{n+k}{(k-1)(k-2)}.$$
The optimal cutoff value that maximizes the expected payoff $E^{BW,P}_{n}$ and this maximum expected payoff are determined in following result.

\begin{teor}
Given an integer $n>1$, let us consider the function $E^{BW,P}_n(r)$ defined above for every integer $1<r<n$ and let $\mathcal{M}(n)$ be the value for which the function $E^{BW,P}_n$ reaches its
maximum. Also let
$$\vartheta:=\frac{1}{2\,W(\frac{e^{\frac{3}{2}}}{2})}=0.552001\dots$$
be the solution to the equation $1 - 3\,x - 2\,x\,\log(x)=0$. Then, the following hold:
\begin{itemize}
\item[i)] $\displaystyle\lim_{n} {\mathcal{M}(n)}/{n}=\vartheta$.
\item[ii)] $\displaystyle \displaystyle \lim_{n}E^{BW,P}_n(\mathcal{M}(n))=\lim_{n}E^{BW,P}_n(\lfloor n\vartheta \rfloor)=\vartheta(1+\vartheta)=0.8567\dots$
\end{itemize}
\end{teor}
\begin{proof}
First, observe that
\begin{align*}
E^{BW,P}_n(r)  &  =\frac{2r(r-1)}{n^{2}}\sum_{k=r+1}^{n}\frac{(n+k)}{(k-1)(k-2)}=\\
&  =2\frac{r}{n}\left(  1+\frac{2}{n}\right)  -2\frac{r}{n}\frac{r-1}%
{n}\left(  1+\frac{3}{n-1}\right)  -2\frac{r}{n}\frac{r-1}{n}\sum_{i=r}%
^{n-1}\frac{1}{i}.
\end{align*}
Now, we can extend $E^{BW,P}_n$ to a real variable function by
$$E^{BW,P}_n(r)=2\frac{r}{n}\left(  1+\frac{2}{n}\right)  -2\frac{r}{n}\frac{r-1}
{n}\left(  1+\frac{3}{n-1}\right)  -2\frac{r}{n}\frac{r-1}{n}(\psi(n)-\psi(r)).$$
Furthermore, it can be seen that the sequence of functions $g_{n}(x):=E^{BW,P}_n(nx)$ converges uniformly on $[0,1]$ to $g(x)=-2x\left(-1+x+x\log x\right)$.

To conclude the proof it is enough to apply Proposition \ref{conv} together with some straightforward computations.
\end{proof}

So far, we have considered the Best-or-Worst variant in which the goal is to select either the best or the worst candidate, indifferent between the two cases. To finish this section we are going to further modify the Best-or-Worst variant. In particular we are going to consider different payoff depending on whether we select the best or the worst candidate. In paticular we are going to consider the following payoff function, with $m<M$.
\begin{equation}\label{poun}
p_U(k)=\begin{cases} m, & \textrm{if the $k$-th candidate is the worst candidate};\\
M, & \textrm{if the $k$-th candidate is the best candidate};\\
0, & \textrm{otherwise}.
\end{cases}
\end{equation}

In this new setting the optimal strategy has two thresholds, as stated in the following result, whose proof is analogue to that of Theorem \ref{BWS}.

\begin{teor}
For the Best-or-Worst variant, if $n$ is the number of candidates and the payments for selecting the worst and the best candidates are, respectively, $m<M$, there exist
$r(n)\leq s(n),$ such that the following strategy is optimal:
\begin{enumerate}
\item Reject the $r(n)$ first interviewed candidates.
\item Accept the first candidate which is better than all the preceding ones until reaching the $s(n)$-th candidate.
\item After that, accept the first candidate which is either better or worse than all the preceding ones.
\end{enumerate}
\end{teor}

Now, let $n$ be the number of candidates and let us consider cutoff values $1<r<s<n$. Then, if $k\in (r,n]$ is an integer, the probability of successfully selecting the best candidate in the $k$-th interview is given by
$$P^{BW,U}_{n,r,s}(k)=\begin{cases}
\frac{r}{(k-1) n}, & \textrm{if $r<k<s$};\\
\frac{r}{k-1}\frac{s-1}{k-2}\frac{1}{n}, & \textrm{if $k\geq s$}.
\end{cases}$$

On the other hand, if $k\in (r,n]$ is an integer, the probability of successfully selecting the best or the worst candidate in the $k$-th interview is given by
$$\overline{P}^{BW,U}_{n,r,s}(k)=\begin{cases}
0, & \textrm{if $r<k<s$};\\
\frac{r}{k-1}\frac{s-1}{k-2}\frac{1}{n}, & \textrm{if $k\geq s$}.
\end{cases}$$
Because, according to the optimal strategy we can only select the worst candidate if $k\geq s$.

Consequently, the expected payoff with $n$ candidates and cutoff values $r<s$ is given by
\begin{align*}
E^{BW,U}_{n}(r,s)&:=\sum_{k=r+1}^n MP^{BW,U}_{n,r,s}(k)+m\overline{P}^{BW,U}_{n,r,s}(k)=\\&=\sum_{k=r+1}^{s}\frac{M\,r}{\left(k-1\right)
\,n}+\sum_{k=s+1}^{n}\left(M+m\right)  \frac{r(s-1)}{(k-1)(k-2)n}.
\end{align*}

The following result determines the cutoff values as well as the corresponding maximum expected payoff.

\begin{teor}\label{TnM}
Given a positive integer $n>2$, let us consider the function
$E^{BW,U}_{n}(r,s)$
defined above for every pair of integers in the set $\{(r,s)\in\mathbb{Z}^{2}:0\leq r\leq s<n\}$ and let $(\mathcal{M}_{1}(n),\mathcal{M}_{2}(n))$ be the point for which $E_n^{BW,U}$ reaches its maximum. Then,
\begin{itemize}
\item[i)] $\displaystyle \lim_{n}\frac{\mathcal{M}_1(n)}{n}=\frac{e^{-1+\frac{n}{M}}M}{m+M}.$
\item[ii)] $\displaystyle \lim_{n}\frac{\mathcal{M}_2(n)}{n}=\frac{M}{m+M}.$
\item[iii)] $\displaystyle \lim_{n}E^{BW,U}_{n}(\mathcal{M}_{1}(n),\mathcal{M}_{2}(n))=\frac{e^{-1+\frac{n}{M}}M^2}{m+M}.$
\end{itemize}
\end{teor}
\begin{proof}
Let us define the sequence of functions $\{h_n\}$ by $h_n(x,y)=E^{BW}_n(nx,xy)$. Then,
$$
\lim_{n}h_n(x,y)=h(x,y)=
\begin{cases}
(M+m)x-(M+m)xy+Mx\log(y/x), & \textrm{if $x,y\neq0$};\\
0 & \textrm{otherwise}.
\end{cases}
$$
and the convergence is uniform on $T:=\{(x,y)\in\mathbb{R}^{2}:0\leq x\leq y\leq1\}$.

Hence, we can apply Proposition \ref{conv2}. To do so, observe that $h$ is a concave function on the convex set $T$ with a negative definite hessian matrix. Since $h$ has only one critical point, namely
$$\left(\frac{e^{-1+\frac{m}{M}}M}{M+m},\frac{M}{M+m}\right)$$
and
$$h\left(\frac{e^{-1+\frac{m}{M}}M}{M+m},\frac{M}{M+m}\right)=\frac{e^{-1+\frac{m}{M}}M^{2}}{M+m}$$
the result follows.
\end{proof}

This result means that the optimal strategy in this setting consists in rejecting roughly the first $n\dfrac{e^{-1+\frac{m}{M}}M}{M+m}$ interviewed candidates, then accepting the first candidate
which is better than all the preceding ones until reaching roughly the $n\dfrac{M}{M+m}$ candidate and, finally accepting the first candidate which is either better or worse than all the preceding ones. Following this strategy, the maximum expected payoff is asymptotically equal to $\displaystyle\frac{e^{-1+\frac{m}{M}}M^{2}}{M+m}$.

\begin{rem}
If $m\ll M$ the cuttof values obtained in Theorem \ref{TnM} are, approximately, $ne^{-1}$ and $n$. This means that the optimal strategy ignores the objective of obtaining the worst candidate and we recover the original secretary problem. In addition, if $m=M$, then both cutoff values coincide with $n/2$ and we recover the original Best-or-Worst variant.
\end{rem}

\section{The Postdoc variant}

In this section we focus on the Postdoc variant, as described in the introduction, in which the goal is to select the second best candidate. First of all we have to prove that, just like in classic problem, the optimal strategy is a threshold strategy.

In this variant it is not obvious that the optimal strategy has only one threshold. This is because the candidate considered in a given interview could be selected both if it is better or the second better than all the preceding ones and in both cases it could end up being the second best candidate. However, we are going to see that selecting a candidate which is better than all the preceding ones is never preferable to waiting for a candidate which is the second better than all the preceding ones.

Assume for a moment that we are following a threshold strategy. Let $n$ be the number of candidates and let us consider a cutoff value $r\in(1,n)$. If $k\in (r,n]$ is an integer, the probability of successfully selecting the best or the worst candidate in the $k$-th interview is $P^{PD}_{n,r}(k)=\frac{r}{k-1}\frac{1}{k}\frac{\binom{k}{2}}{\binom{n}{2}}$. Thus, the probability function of succeeding in the Postdoc variant with $n$ candidates using $r$ as cutoff value and provided we are following a threshold strategy for the second best candidate, is given by
$$F^{PD}_{n}(r):=\sum_{k=r+1}^{n}P^{PD}_{n,r}(k)=\sum_{k=r+1}^{n}\frac{r\,{\binom{k}{2}}}{\left(  -1+k\right)  \,k\,{\binom{n}{2}}}.$$
Note that the following holds:
\begin{align*}
F^{PD}_{n}(r):=&=\sum_{k=r+1}^{n}\frac{r\,{\binom{k}{2}}}{\left(  -1+k\right)  \,k\,{\binom
{n}{2}}}=\frac{r\,{\binom{r+1}{2}}}{\left(  -1+r+1\right)  \,(r+1)\,{\binom
{n}{2}}}+\sum_{k=r+2}^{n}\frac{r\,{\binom{k}{2}}}{\left(  -1+k\right)  \,k\,{\binom
{n}{2}}}=\\
&= \frac{ \,{\binom{r+1}{2}}}{  \,(r+1)\,{\binom
{n}{2}}}+\sum_{k=r+2}^{n}\frac{(r+1)r\,{\binom{k}{2}}}{\left(  -1+k\right)  \,k\,{(r+1)\binom
{n}{2}}}=\\
&= \frac{ \,{\binom{r+1}{2}}}{  \,(r+1)\,{\binom
{n}{2}}}+\frac{r}{r+1}\sum_{k=r+2}^{n}\frac{(r+1) \,{\binom{k}{2}}}{\left(  -1+k\right)  \,k\,{ \binom
{n}{2}}}=\\
&= \frac{ \,{\binom{r+1}{2}}}{  \,(r+1)\,{\binom
{n}{2}}}+ \frac{r}{r+1}F^{PD}_{n}(r+1).
\end{align*}
On the other hand, let us denote by $T_n(r)$ the probability of success after the $r$-th interview provided we have already selected a candidate which is better than all the preceding ones. Then, the probability of finding the second best candidate in the $(r+1)$-th interview is $\frac{1}{r+1}$ and, furthermore, the probability of not finding a better candidate among all the remaining interviews is $\frac{\binom{r+1}{2}}{\binom{n}{2}}$. On the other hand, the probability of not obtaining the second best candidate in the $(r+1)$-th interview is $\frac{r}{r+1}$ and the probability of success in this case will be $T_n(r+1)$. Hence,
$$T_n (r)=\frac{1}{r+1}\frac{\binom{r+1}{2}}{ \binom{n}{2}}+\frac{r}{r+1}T_n(r+1).$$
Thus, we have seen that $T_n(r)$ and $F^{PD}_{n}(r) $ both satisfy the same recurrence relation in $r$. Moreover, it holds that $T_n(n-1)=F^{PD}_{n}( n-1)=1/n$ so, consequently, we obtain that $T_n(r)=F^{PD}_{n}(r) $ for every $r<n$.

Note that this means that the optimal strategy can neglect if a given candidate is better than all the preceding ones and focus only on whether the candidate is the second better than all the preceding ones and thus the optimal strategy has only one threshold.

\begin{teor}\label{TEORPDC}
For the Postdoc variant, if $n$ is the number of candidates, there exists $r(n)$ such
that the following strategy is optimal:
\begin{enumerate}
\item Reject the $r(n)$ first interviewed candidates.
\item After that, accept the first candidate which is the second best until then.
\end{enumerate}
\end{teor}
\begin{proof}
Just use the same ideas as in Theorem \ref{BWS}.
\end{proof}

 Thus, the probability function of succeeding in the Postdoc variant with $n$ candidates using $r$ as cutoff value, is given by
$$F^{PD}_{n}(r):=\sum_{k=r+1}^{n}P^{PD}_{n,r}(k)=\frac{r(n-r)}{n(n-1)}.$$

Observe that we have obtained that $F_n^{PD}(r)=\dfrac{1}{2}F_n^{BW}(r)$. Consequently, if we follow the previous strategy, the optimal cutoff value is the same as in the Best-or-Worst variant; i.e., $\lfloor \frac{n}{2}\rfloor$) and the maximum probability of success is one half of the maximum probability of success in the Best-or-Worst variant (see Theorem \ref{BWP}).

We are now going to consider the Postdoc variant with the payoff function $p_C$ given in (\ref{pocost}); i.e., we assume that performing each interview has a constant cost of $1/n$. Under this assumption it can be proved that the optimal strategy has two thresholds.

\begin{teor}\label{PDP}
For the Postdoc variant, if $n$ is the number of candidates and if the payoff function is given by (\ref{pocost}), there exist $r(n)\leq s(n)$, such that the following strategy is optimal:
\begin{enumerate}
\item Reject the $r(n)$ first interviewed candidates.
\item Accept the first candidate which is better than all the preceding ones until reaching the $s(n)$-th candidate.
\item After that, accept the first candidate which is either better or second better than all the preceding ones.
\end{enumerate}
\end{teor}
\begin{proof}
Proceed as in Theorem \ref{BWS} with each threshold separately.
\end{proof}

Under this strategy, the probability of successfully selecting the second best candidate in the $k$-th interview is given by the function
$$P^{PD,C}_{n,r,s}(k)=\begin{cases}
\frac{r(n-k)}{n(n-1)(k-1)}, & \textrm{if $r<k<s$};\\
\frac{r(s-1)(n-k)}{n(n-1)(k-1)(k-2)}+\frac{r(s-1)}{n(n-1)(k-2)}, & \textrm{if $k\geq s$}.
\end{cases}$$

Consequently, the expected payoff with $n$ candidates and cutoff values $r<s$ is given by
$$E^{PD,C}_{n}(r,s)=\sum_{k=r+1}^n \left(1-\frac{k}{n}\right)P^{PD,C}_{n,r,s}(k).$$
In the following result we determine the optimal cutoff values and the maximum expected payoff.

\begin{teor}
Given a positive integer $n>2$ let us consider the function
$E^{PD,C}_{n}(r,s)$ defined above for every $(r,s)\in\{(r,s) \in\mathbb{Z}^{2}:0\leq r\leq s<n\}$ and let $(\mathcal{M}_{1}
(n),\mathcal{M}_{2}(n))$ be the point for which $E^{PD,C}_{n}$ reaches its maximum. Then,
\begin{itemize}
\item[i)] $\displaystyle \lim_{n}  {\mathcal{M}_1(n)}/{n}=0.17248\dots$
\item[ii)] $\displaystyle \lim_{n}  {\mathcal{M}_2(n)}/{n}=0.39422\dots$
\item[iii)] $\displaystyle \lim_{n}E^{PD,C}_{n}(\mathcal{M}_{1}(n),\mathcal{M}_{2}(n))=0.11811\dots$
\end{itemize}
\end{teor}
\begin{proof}
First of all, observe that
\begin{align*}
E^{PD,C}_{n}(r,s)&=\frac{r}{n^{2}}\left(  n+\frac{n-1}{s-1}-s+\frac{\left(
s-r\right)  \,\left(  3-4\,n+r+s\right)  }{2\,\left(  n-1\right)  }\right)  +\\ &
+\frac{r}{n^{2}}\left(  \left(  1-s\right)  \,\psi(-1+n)-\left(  n-1\right)
\,\psi(r)+\left(  n-2+s\right)  \,\psi(s-1)\right).
\end{align*}
Thus, if we define the sequence of functions $\{h_n\}$ by $h_n(x,y)=E^{PD,C}_n(nx,ny)$, it follows that
$$\lim_{n}h_n(x,y)=h(x,y):=\begin{cases}
\frac{x\left(2-6y+y^{2}+4x-x^{2}+2(1+y)\log
y-2\,\log x\right)}{2}, & \textrm{if $x,y\neq0$};\\
0, & \textrm{otherwise}.
\end{cases}$$
and the convergence is uniform on $\{(x,y)\in\mathbb{R}^{2}:0\leq x\leq y\leq1\}$.

Using elementary techniques we get that $h$ reaches its absolute maximum at the point $(\alpha,\beta)$ with $\beta:=0.39422\dots$ is the solution to $-2+\frac{1}{\beta}+\beta+\log(\beta)=0$ and $\alpha:=0.1724844\dots$ is the solution to $1-\frac{1}{\beta}-2\,\beta-\frac{\beta^{2}}{2}+4\,\alpha-\frac{3\,\alpha^{2}}{2}-\log(\alpha)=0$.

The fact that $h(\alpha,\beta)=0.11811\dots$ concludes the proof.
\end{proof}

Finally, let us consider the Postdoc variant with the payoff function $p_P$ given in (\ref{popay}); i.e., we assume that performing each interview has an additional payoff of $1/n$. Under this assumption, it is clear that no optimal strategy will accept a candidate which is better than all the preceding ones because, if the search continues, the probability of success is the same and the payoff will be greater. Hence, we must only consider strategies with one threshold for the second best candidate, as in Theorem \ref{TEORPDC}, ignoring if the interviewed candidate in better than the preceding ones. In this setting, the expected payoff with $n$ candidates and cutoff value $r$ is given by
$$E_n^{PD,P}(r):=\sum_{k=r+1}^n\left(1+\frac{k}{n}\right)P^{PD}_{n,r}(k)=\frac{r(n-r)(3n+1+r)}{2n^2(n-1)}.$$
The optimal cutoff value that maximizes the expected payoff $E^{PD,P}_{n}$ and this maximum expected payoff are determined in the following result.

\begin{teor}
Given an integer $n>1$, let us consider the function $E_n^{PD,P}(r)$ defined above for every integer $1<r<n$ and let $\mathcal{M}(n)$ be the value for
which the function $E_n^{PD,P}$ reaches its maximum. Then, the following hold:
\begin{itemize}
\item[i)] $\displaystyle \lim_{n}\frac{\mathcal{M}(n)}{n}=\frac{\sqrt{13}-2}{3}=0.53518\dots$
\item[ii)] $\displaystyle\lim_{n}E_n^{PD,P}(\mathcal{M}(n))=\frac{13\sqrt{13}-35}{27}=0.4397\dots$
\end{itemize}
\end{teor}
\begin{proof}
Since $E_n^{PD,P}$ is a degree 3 polynomial, we can explicitly obtain the exact value of $\mathcal{M}(n)$ by elementary methods. Namely,
$$\mathcal{M}(n)=\frac{-1-2\,n+\sqrt{1+7\,n+13\,n^{2}}}{3}.$$
The result follows immediately.
\end{proof}

\begin{rem}
Note that we can further refine the previous result by noting that $\displaystyle \mathcal{M}(n)=\left(\frac{\sqrt{13}-2}{3}\right)n+\frac{7-2\sqrt{13}}{6\sqrt{13}}+o(n)$. In this case, $[\mathcal{M}(n)]$ is the optimal cutoff value for all $n$ up to 10000, without any exception.
\end{rem}

\section{Conclusions}

In this paper, we have analyzed two variants of the secretary problem which happen to be closely related: the Postdoc and the Best-or-Worst variants. Both of them have the same optimal threshold strategy and the mean payoff for the first one is twice as for the second one.

We now show a comparative table of the asymptotic optimal cutoff value (ACV)
given by $\displaystyle \lim_{n} \mathcal{M}(n)/n$ and the the asymptotic maximum expected payoff (AMP) in the classical
secretary problem, in the Best-or-Worst variant and in the Postdoc variant with payoff functions $p_B$, $p_C$ and $p_P$. In the case of the Postdoc variant with payoff function $p_P$, in the cell corresponding to $\mathcal{M}(n)/n$ we show the two thresholds related to the optimal strategy in that setting.

\[%
\begin{tabular}
[c]{|c|c|c|c|c|c|c|}\hline
{\footnotesize Payoff} & \multicolumn{2}{|c|}{Classic} &
\multicolumn{2}{|c|}{Best-or-Worst} & \multicolumn{2}{|c|}{Postdoc}%
\\\cline{2-7}
& $\text{ACV}$ & $\text{AMP}$ & $\text{ACV}$ & $\text{AMP}$ &
$\text{ACV}$ & $\text{AMP}$\\\hline
$p_B$ & $e^{-1}$ & $e^{-1}$ & $1/2$ & 1/2 & $1/2$ & 1/4\\\hline
$p_C$ & $%
\begin{array}
[c]{c}%
\rho\simeq\\
0.2031
\end{array}
$ & $%
\begin{array}
[c]{c}%
\rho-\rho^{2}\simeq\\
0.1619
\end{array}
$ & $%
\begin{array}
[c]{c}%
\theta\simeq\\
0.2846
\end{array}
$ & $%
\begin{array}
[c]{c}%
\theta-\theta^{2}\simeq\\
0.2036
\end{array}
$ & $%
\begin{array}
[c]{c}%
0.1724,\\
0.3942
\end{array}
$ & $0.1181$\\\hline
$p_P$ & $%
\begin{array}
[c]{c}%
\eta\simeq\\
0.4263
\end{array}
$ & $%
\begin{array}
[c]{c}%
\eta^{2}+\eta\simeq\\
0.6080
\end{array}
$ & $%
\begin{array}
[c]{c}%
\vartheta\simeq\\
0.5520
\end{array}
$ & $%
\begin{array}
[c]{c}%
\vartheta^{2}+\vartheta\simeq\\
0.8567
\end{array}
$ & $%
\begin{array}
[c]{c}%
\frac{\sqrt{13}-2}{3}\simeq\\
0.5351
\end{array}
\,$ & $%
\begin{array}
[c]{c}%
\frac{13\sqrt{13}-35}{27}\\
\simeq0.4397
\end{array}
$\\\hline
\end{tabular}
\ \
\]

\end{document}